\documentclass[11pt]{article}

\usepackage[T1]{fontenc}
\usepackage{lmodern}
\usepackage{amsmath,amssymb,amsthm}
\usepackage[margin=1in]{geometry}
\usepackage[hidelinks]{hyperref}

\hypersetup{
  pdftitle={A counterexample to the Foregger--Sinkhorn tie-point conjecture},
  pdfauthor={Yair Lavi}
}

\newtheorem{theorem}{Theorem}
\newtheorem{lemma}[theorem]{Lemma}
\theoremstyle{definition}
\newtheorem*{conjecture}{Tie-point conjecture}

\title{A counterexample to the Foregger--Sinkhorn tie-point conjecture}
\author{Yair Lavi}
\date{13 August 2026}

\begin{document}
\maketitle

\begin{abstract}

The Foregger--Sinkhorn tie-point conjecture asserts that if a nearly
decomposable doubly stochastic matrix minimizes the permanent on a face and
the permanental cofactor at a prescribed zero is larger than its permanent,
then that zero is a tie point. We give a counterexample in dimension eight.
\end{abstract}



\section{Introduction}

Let \(\Omega_n\) denote the polytope of \(n\)-by-\(n\) doubly stochastic
matrices. For a set \(Z\) of positions, put

\[
\Omega_n(Z)=\{A=(a_{ij})\in\Omega_n:a_{ij}=0
\text{ for every }(i,j)\in Z\}.
\]

A square nonnegative matrix is \textbf{partly decomposable} if it contains an
\(s\)-by-\((n-s)\) zero submatrix for some \(1\le s<n\), and is \textbf{fully
indecomposable} otherwise. It is \textbf{nearly decomposable} if it is fully
indecomposable but replacing any one of its positive entries by zero makes it
partly decomposable.

Let \(S\) be a nearly decomposable support and let \(g\notin S\) be a zero
position. Hartfiel \cite[p.~224]{hartfiel1971} called \(g\) a \textbf{tie point} if
\((S\cup\{g\})\setminus\{e\}\) is partly decomposable for every original
support position \(e\in S\). Equivalently, \(g\) fails to be a tie point as
soon as \((S\cup\{g\})\setminus\{e\}\) is fully indecomposable for some
\(e\in S\).

For a matrix \(A\), write \(A(i\mid j)\) for the submatrix obtained by
deleting row \(i\) and column \(j\).

Minc states the Foregger--Sinkhorn conjecture as Conjecture 41 in his 1987
survey \cite[p.~135]{minc1987}; Cheon and Wanless retain the same numbering
in their later catalogue \cite[p.~333]{cheonwanless2005}. Its statement is
the following.

\begin{conjecture}
Let \(A\) be a nearly decomposable matrix that minimizes the permanent on
\(\Omega_n(Z)\), and let \((i,j)\in Z\). If
\(\operatorname{per}A(i\mid j)>\operatorname{per}A\), then \((i,j)\) is a tie
point for \(A\).
\end{conjecture}

As Minc records \cite[p.~135]{minc1987}, Foregger and Sinkhorn
\cite{foreggersinkhorn1986} proposed this
implication after giving counterexamples to an earlier conjecture of his.
Cheon and Wanless \cite[p.~333]{cheonwanless2005} report that Foregger
\cite{foregger1987} subsequently verified the tie-point implication for a
special family they call complexes, described as two special vertices joined
by separate paths. The source terminology is not uniform: the catalogue says
``complex,'' whereas the title of \cite{foregger1987} says ``multiplexes.'' We
follow the catalogue terminology and give a counterexample outside its
described family.

The exact order-eight statement appears in Theorem 1. Its proof occupies
Sections 2--5. The result is not a minimality theorem: the absence of smaller
counterexamples has not been proved.

\subsection{The counterexample}

For an \(n\)-by-\(n\) zero-one matrix \(D\), put
\(Z(D)=[n]^2\setminus\operatorname{supp}D\) and
\(\Omega(D)=\Omega_n(Z(D))\). Thus \(\Omega(D)\) is the closed face of doubly
stochastic matrices whose support is contained in \(D\). Consider the support

\[
D_8=\begin{pmatrix}
1&1&1&1&0&0&0&0\\
1&0&0&0&0&0&1&0\\
0&1&0&0&0&0&1&0\\
0&0&0&0&1&1&1&0\\
0&0&1&0&0&0&0&1\\
0&0&0&1&0&0&0&1\\
0&0&0&0&1&0&0&1\\
0&0&0&0&0&1&0&1
\end{pmatrix}.
\]

Its row strings are

\begin{verbatim}
11110000|10000010|01000010|00001110|
00100001|00010001|00001001|00000101.
\end{verbatim}

\begin{theorem}
The cubic
\(7x^3-13x^2+12x-4\) has a unique root \(\beta\) satisfying
\(59/100<\beta<3/5\). Let \(Z_8=Z(D_8)\).

The support \(D_8\) is nearly decomposable, the matrix below has support
exactly \(D_8\), and the permanent on the entire closed face
\(\Omega(D_8)=\Omega_8(Z_8)\) has the unique global minimizer

\[
A_\beta=\begin{pmatrix}
\beta/2&\beta/2&(1-\beta)/2&(1-\beta)/2&0&0&0&0\\
1-\beta/2&0&0&0&0&0&\beta/2&0\\
0&1-\beta/2&0&0&0&0&\beta/2&0\\
0&0&0&0&\beta/2&\beta/2&1-\beta&0\\
0&0&(1+\beta)/2&0&0&0&0&(1-\beta)/2\\
0&0&0&(1+\beta)/2&0&0&0&(1-\beta)/2\\
0&0&0&0&1-\beta/2&0&0&\beta/2\\
0&0&0&0&0&1-\beta/2&0&\beta/2
\end{pmatrix}.
\]

Its permanent is

\[
m_8=\operatorname{per}A_\beta
=\frac{13233\beta^2-10268\beta+2536}{67228}
\approx0.01628247758894465.
\]

At the prescribed zero \(g=(1,5)\in Z_8\),

\[
\operatorname{per}A_\beta(1\mid5)-\operatorname{per}A_\beta
=\frac{286+1599\beta-3232\beta^2}{9604}
>\frac{2047}{240100}>0.
\]

But \(g\) is not a tie point. Consequently the Foregger--Sinkhorn tie-point
conjecture is false.
\end{theorem}

\section{The support is nearly decomposable}

We use the standard bipartite graph of a zero-one support: row vertices and
column vertices are joined at the support positions. A square support is fully
indecomposable precisely when this graph is connected and every edge lies in
a perfect matching. The second condition is usually called total support;
Hartfiel records the relevant total-support facts and single-edge deletion
lemma \cite[pp.~223--224]{hartfiel1971}.

Write \(ij\) for the edge joining row \(i\) to column \(j\). The complete
list of perfect matchings of \(D_8\), written as the selected column in rows
\(1,\ldots,8\), is

\[
17253486,\quad17263458,\quad21753486,\quad21763458,
\quad31278456,\quad41273856.                 \tag{2.1}
\]

This list is exhaustive by splitting on the four possible choices in row 1;
the first two choices give two completions each, and the last two give one
completion each. The union of the six matchings is all 19 edges of \(D_8\).
Moreover,

\[
11,12,13,14,21,32,53,64,27,58,47,78,88,45,46              \tag{2.2}
\]

is a spanning tree of the 16-vertex bipartite graph. Thus \(D_8\) is
connected and has total support, so it is fully indecomposable.

It remains to show that every edge is essential. In the following table, the
right-hand edge \(f\) has the property that every perfect matching containing
\(f\) also contains the edge \(e\) in the left column.

\begin{center}
\begin{tabular}{cc}
\hline
deleted edge \(e\) & edge \(f\) left without a perfect matching \\
\hline
11 & 27 \\
12 & 37 \\
13 & 58 \\
14 & 68 \\
21 & 12 \\
27 & 11 \\
32 & 11 \\
37 & 12 \\
45 & 78 \\
46 & 88 \\
47 & 13 \\
53 & 11 \\
58 & 13 \\
64 & 11 \\
68 & 14 \\
75 & 13 \\
78 & 45 \\
86 & 13 \\
88 & 46 \\
\hline
\end{tabular}
\end{center}

The assertion in every row is checked directly from the exhaustive list
(2.1). After deleting \(e\), the distinct edge \(f\) remains but lies in no
perfect matching. Hence the resulting support does not have total support and,
by Hartfiel's lemma \cite[p.~224]{hartfiel1971}, is partly decomposable. This
proves that \(D_8\) is nearly decomposable.

The incidence pattern also separates this support from the quoted complex
case. Row 1 and column 8 are each incident with four support edges, while row
4 and column 7 are each incident with three; every other vertex is incident
with two. Thus \(D_8\) has four vertices where more than two support edges
meet, rather than the two special vertices in the catalogue description of a
complex.

\section{The complete face and its invariant permanent}

The global argument must include every boundary stratum of the closed face.
Let \(a,b,c,d,z,h\ge0\) satisfy

\[
a+b=c+d=:t,\qquad z+h=1-t,\qquad 0\le t\le1.             \tag{3.1}
\]

These constraints also ensure that every complementary entry displayed below
is nonnegative.

Then every matrix in \(\Omega(D_8)\), and no other matrix, has the unique form

\[
A(a,b,c,d,z,h)=
\begin{pmatrix}
a&b&z&h&0&0&0&0\\
1-a&0&0&0&0&0&a&0\\
0&1-b&0&0&0&0&b&0\\
0&0&0&0&c&d&1-t&0\\
0&0&1-z&0&0&0&0&z\\
0&0&0&1-h&0&0&0&h\\
0&0&0&0&1-c&0&0&c\\
0&0&0&0&0&1-d&0&d
\end{pmatrix}.                                           \tag{3.2}
\]

To verify this claim directly, start with the entries \(a,b,z,h\) in row 1
and \(c,d\) in row 4. The sum equations for column 1 and row 2 force the
remaining weights \(1-a,a\); those for column 2 and row 3 force \(1-b,b\).
Likewise, columns 3, 4, 5, 6 together with rows 5, 6, 7, 8 force, respectively,
the pairs

\[
(1-z,z),\qquad(1-h,h),\qquad(1-c,c),\qquad(1-d,d).
\]

If \(x\) denotes the remaining entry in row 4 and column 7, the four remaining
row- and column-sum equations, for row 1, column 7, row 4, and column 8, are

\[
a+b+z+h=1,\qquad a+b+x=1,
\qquad c+d+x=1,\qquad c+d+z+h=1.
\]

They give

\[
a+b=c+d=:t,\qquad z+h=1-t,\qquad x=1-t.
\]

This proves that (3.1)--(3.2) parameterize the entire closed face, including
every boundary stratum.

Regard (3.2) as a weighted bipartite graph with vertices \(R_i,C_j\), and give
the edge \(R_iC_j\) the corresponding matrix-entry weight. The permanent is
the sum of the weights of all perfect matchings. Apart from the edge
\(R_4C_7\), the graph consists of the following six paths of length three:

\[
\begin{array}{c|c}
\text{path}&\text{edge weights}\\ \hline
R_1-C_1-R_2-C_7&a,\ 1-a,\ a\\
R_1-C_2-R_3-C_7&b,\ 1-b,\ b\\
R_4-C_5-R_7-C_8&c,\ 1-c,\ c\\
R_4-C_6-R_8-C_8&d,\ 1-d,\ d\\
R_1-C_3-R_5-C_8&z,\ 1-z,\ z\\
R_1-C_4-R_6-C_8&h,\ 1-h,\ h.
\end{array}
\]

Call each pair of these paths having the same endpoints a \textbf{bundle}. Thus
there are three bundles, with path labels \((a,b)\), \((c,d)\), and \((z,h)\).
The six paths meet at \(R_1,C_7,R_4,C_8\); we call these four vertices the
\textbf{hubs}.

On any one of the listed paths, the row- and column-sum equations make the
two endpoint-edge weights equal. We call their common value \(r\) the
\textbf{path parameter}. The three edge weights along that path are then
\(r,1-r,r\).

Fix a perfect matching of the full graph and consider its restriction to one
bundle. All four internal vertices of the bundle must be covered exactly once.
For example, the \((a,b)\)-bundle joins \(R_1\) to \(C_7\), and its internal
vertices are \(C_1,R_2,C_2,R_3\). If the matching uses neither hub through
this bundle, it must select the two middle edges \(C_1R_2\) and \(C_2R_3\),
whose product of weights is \((1-a)(1-b)\). If it uses both hubs through this
bundle, there are exactly two possibilities:

\[
\{R_1C_1,R_2C_7,C_2R_3\}
\quad\text{or}\quad
\{C_1R_2,R_1C_2,R_3C_7\},
\]

with products of weights \(a^2(1-b)\) and \((1-a)b^2\), respectively. A
configuration using exactly one hub is impossible: on either length-three
path, covering its two internal vertices forces either its middle edge alone
or both endpoint edges.

More generally, let a bundle have path parameters \(r,q\), and put
\(s=r+q\) and \(p=rq\). The total contribution of configurations using
neither hub is

\[
G_s(p)=(1-r)(1-q)=1-s+p.
\]

The total contribution of the two configurations using both hubs is

\[
\begin{aligned}
F_s(p)
&=r^2(1-q)+q^2(1-r)\\
&=s^2-(s+2)p.
\end{aligned}                                             \tag{3.3}
\]

Put

\[
u=ab,\qquad v=cd,\qquad w=zh.                             \tag{3.4}
\]

The sums and products of the path parameters in the three bundles are
\((t,u)\), \((t,v)\), and \((1-t,w)\), respectively. Recording only which
hubs are joined, these bundles give the connections \(R_1C_7\), \(R_4C_8\),
and \(R_1C_8\). Together with the single edge \(R_4C_7\), they form the cycle

\[
R_1-C_7-R_4-C_8-R_1.
\]

Consequently a perfect matching has only two global matching states.

\begin{itemize}
\item If \(R_4C_7\) is unused, the \((a,b)\)- and \((c,d)\)-bundles are in
state \(F\), while the \((z,h)\)-bundle is in state \(G\).
\item If \(R_4C_7\) is used, it contributes \(1-t\); the first two bundles
are in state \(G\), and the third is in state \(F\).
\end{itemize}

Summing the contributions of these two disjoint classes of perfect matchings
gives

\[
P_8(t,u,v,w)
=F_t(u)F_t(v)G_{1-t}(w)
 +(1-t)G_t(u)G_t(v)F_{1-t}(w).                           \tag{3.5}
\]

For nonnegative \(r,q\) with \(r+q=s\),

\[
0\le rq\le\frac{s^2}{4},                                \tag{3.6}
\]

because \((r-q)^2\ge0\). Equality at the upper endpoint holds exactly when
\(r=q=s/2\). It follows that

\[
0\le u,v\le\frac{t^2}{4},
\qquad 0\le w\le\frac{(1-t)^2}{4}.
\]

For each fixed \(t\), these inequalities define a three-dimensional box in
\((u,v,w)\)-space, possibly degenerate at \(t=0\) or \(t=1\). Normalize it by

\[
u=U\frac{t^2}{4},\qquad v=V\frac{t^2}{4},\qquad
w=W\frac{(1-t)^2}{4},\qquad 0\le U,V,W\le1.             \tag{3.7}
\]

In these normalized coordinates, the six bundle factors in (3.5) are

\[
\begin{aligned}
F_t(u)&=t^2\left(1-\frac{t+2}{4}U\right),
&G_t(u)&=1-t+\frac{t^2}{4}U,\\
F_t(v)&=t^2\left(1-\frac{t+2}{4}V\right),
&G_t(v)&=1-t+\frac{t^2}{4}V,\\
F_{1-t}(w)&=(1-t)^2\left(1-\frac{3-t}{4}W\right),
&G_{1-t}(w)&=t+\frac{(1-t)^2}{4}W.
\end{aligned}
\]

The formula (3.5) is affine in each of \(u,v,w\) separately. Successive
one-variable interpolation in the three normalized coordinates therefore
expresses every value as a convex combination of the eight corner values.
Specifically, if \(C_{ijk}(t)\) is the value of (3.5) at
\((U,V,W)=(i,j,k)\), then

\[
P_8(t,u,v,w)=\sum_{(i,j,k)\in\{0,1\}^3}
C_{ijk}(t)
U^i(1-U)^{1-i}V^j(1-V)^{1-j}W^k(1-W)^{1-k}.             \tag{3.8}
\]

All eight weights are nonnegative, and their sum is
\((U+(1-U))(V+(1-V))(W+(1-W))=1\). This conclusion is a
special consequence of the multiaffine identity (3.8), not a claim that the
permanent is a convex function. A corner coordinate \(0\) means that the
corresponding product is zero; when the prescribed sum is positive, one path
parameter is zero and the other equals that sum. A coordinate \(1\) means
that the product is maximal, so the two path parameters are equal. At \(t=0\)
or \(t=1\), normalized coordinates on collapsed intervals may be assigned
arbitrarily; equivalently, (3.8) at those endpoints follows by continuity.

\section{Unique global minimum on the closed face}

We first bound the seven corners other than \(C_{111}\), and then treat
\(C_{111}\) separately.

\begin{lemma}
For every \(t\in[0,1]\), \(C_{ijk}(t)>1/50\) whenever
\((i,j,k)\ne(1,1,1)\).
\end{lemma}

\begin{proof}
Since \(x\mapsto x^5\) is strictly convex on \([0,1]\), Jensen's inequality
gives

\[
\frac{t^5+(1-t)^5}{2}
\ge\left(\frac{t+(1-t)}2\right)^5
=\frac1{32}.
\]

Hence

\[
C_{000}(t)=t^5+(1-t)^5\ge\frac1{16}>\frac1{50},         \tag{4.1}
\]

with equality in the first inequality exactly at \(t=1/2\).

For the corner \(C_{100}=C_{010}\), write

\[
C_{100}(t)=
\underbrace{\frac{t^5(2-t)}4}_{X(t)}
+\underbrace{\frac{(1-t)^4(2-t)^2}{4}}_{Y(t)}.          \tag{4.2}
\]

On \([0,1]\), \(X\) is increasing and \(Y\) is decreasing; indeed,

\[
X'(t)=\frac{t^4(5-3t)}2\ge0,
\qquad
Y'(t)=-\frac{(1-t)^3(2-t)(5-3t)}2\le0.
\]

Both \(X\) and \(Y\) are nonnegative on \([0,1]\). Three elementary estimates
now cover the interval. If \(t\le1/2\), then

\[
Y(t)\ge Y(1/2)=\frac9{256}>\frac1{50}.
\]

If \(t\ge3/5\), then

\[
X(t)\ge X(3/5)=\frac{1701}{62500}>\frac1{50}.
\]

Finally, if \(1/2\le t\le3/5\), then

\[
X(t)\ge\frac3{256}>\frac1{100},
\qquad
Y(t)\ge\frac{196}{15625}>\frac1{100},
\]

and hence \(C_{100}(t)>1/50\). The identity
\(C_{001}(t)=C_{100}(1-t)\) proves the same bound for the remaining asymmetric
corner.

For the two corners \(C_{101}=C_{011}\), exact differentiation gives

\[
\begin{aligned}
C_{101}(t)
&=\frac{(2-t)(1+t)(t^2-t+1)(7t^2-7t+2)}{16},\\
C_{101}'(t)
&=-\frac{(2t-1)(3t^2-3t+2)(7t^2-7t-8)}{16}.
\end{aligned}                                             \tag{4.3}
\]

The quadratic \(3t^2-3t+2\) is positive, while
\(7t^2-7t-8<0\), throughout \([0,1]\). Thus the derivative has the sign of
\(2t-1\), and

\[
C_{101}(t)\ge C_{101}(1/2)=\frac{27}{1024}>\frac1{50}.
\]

The last of these seven corners is

\[
\begin{aligned}
C_{110}(t)
&=\frac{(t-2)^2(7t^4-19t^3+25t^2-16t+4)}{16},\\
C_{110}'(t)
&=\frac{(t-2)(3t-2)(7t-10)(2t^2-3t+2)}{16}.
\end{aligned}                                             \tag{4.4}
\]

On \([0,1]\), the factors \(t-2\) and \(7t-10\) are negative, and
\(2t^2-3t+2\) is positive. The derivative therefore has the sign of
\(3t-2\). Consequently,

\[
C_{110}(t)\ge C_{110}(2/3)=\frac{16}{729}>\frac1{50}.
\]

We have proved

\[
C_{ijk}(t)>\frac1{50}
\qquad\text{for every }(i,j,k)\ne(1,1,1).
\]

This proves the lemma.
\end{proof}

The remaining corner is \(C_{111}\). Here the two path parameters in each
bundle are equal:

\[
a=b=c=d=\frac t2,
\qquad z=h=\frac{1-t}{2}.
\]

Its value is

\[
q_8(t)=C_{111}(t)
=\frac{(t-2)^2(t+1)(8t^4-19t^3+25t^2-16t+4)}{64}.       \tag{4.5}
\]

It crosses below the separating level because

\[
q_8(3/5)-\frac1{50}=-\frac{577}{156250}<0.              \tag{4.6}
\]

Its derivative
factors as

\[
q_8'(t)=\frac{(t-2)(4t^2-3t-4)}{32}
\bigl(7t^3-13t^2+12t-4\bigr).                           \tag{4.7}
\]

On \([0,1]\), both \(t-2\) and \(4t^2-3t-4\) are negative, so the
prefactor in (4.7) is positive. Let

\[
g(t)=7t^3-13t^2+12t-4.
\]

The derivative \(g'(t)=21t^2-26t+12\) is positive on the real line: its
leading coefficient is positive and its discriminant is \(-332\). Moreover,

\[
g(59/100)=-\frac{7647}{1000000}<0,
\qquad g(3/5)=\frac4{125}>0.                             \tag{4.8}
\]

Thus \(g\) has a unique root \(\beta\in(59/100,3/5)\). Since the prefactor
in (4.7) is positive, \(q_8\) strictly decreases up to \(\beta\) and strictly
increases after it. Hence \(q_8(t)\ge q_8(\beta)\) on \([0,1]\), with equality
only at \(t=\beta\).

Let \(m=q_8(\beta)\). By (4.6),

\[
m=q_8(\beta)\le q_8(3/5)<\frac1{50}.
\]

For an arbitrary point of the closed face, the multiaffine identity (3.8)
expresses its permanent as a convex combination of the eight values
\(C_{ijk}(t)\). By Lemma 2, seven of these values are strictly larger than
\(1/50>m\); the remaining value is
\(C_{111}(t)=q_8(t)\ge q_8(\beta)=m\). Therefore every matrix in the face has
permanent at least \(m\).

The equality conditions are equally explicit. Equality requires zero weight
on every one of the seven corners other than \(C_{111}\). Since all eight
weights in (3.8) are nonnegative and sum to one, the weight on \(C_{111}\)
must be one; hence \(U=V=W=1\). Equality also requires \(q_8(t)=m\), so the
uniqueness of the minimum of \(q_8\) gives \(t=\beta\). Finally, maximal
product at a fixed pair sum is attained uniquely by equal numbers. Thus

\[
a=b=c=d=\frac\beta2,
\qquad z=h=\frac{1-\beta}{2}.                            \tag{4.9}
\]

It follows that the displayed matrix \(A_\beta\) is the unique global
minimizer on the entire closed face. Reducing (4.5) modulo \(g(t)\) gives

\[
m=q_8(\beta)
=\frac{13233\beta^2-10268\beta+2536}{67228},            \tag{4.10}
\]

as asserted in Theorem 1.

\section{The prescribed zero is not a tie point}

Let \(g=(1,5)\), a zero position of \(D_8\). Direct permanent expansion of
the \(7\)-by-\(7\) minor gives

\[
\operatorname{per}A_\beta(1\mid5)
=\frac{\beta(\beta-2)^3(\beta-1)(\beta+1)^2}{64}.       \tag{5.1}
\]

Subtracting (4.5)
and reducing modulo \(7\beta^3-13\beta^2+12\beta-4\) yields

\[
\operatorname{per}A_\beta(1\mid5)-\operatorname{per}A_\beta
=H(\beta),
\qquad
H(t)=\frac{286+1599t-3232t^2}{9604}.                    \tag{5.2}
\]

Now \(H'(t)=(1599-6464t)/9604\), and

\[
H'(59/100)=-\frac{55369}{240100}<0.
\]

Since \(H'\) is decreasing, \(H\) is strictly decreasing throughout the
isolating interval. Therefore

\[
H(\beta)>H(3/5)=\frac{2047}{240100}>0.                  \tag{5.3}
\]

It remains to show that \(g\) is not a tie point. Add \((1,5)\) and delete
the original edge \((4,5)\). The surviving support is

\begin{verbatim}
11111000|10000010|01000010|00000110|
00100001|00010001|00001001|00000101.
\end{verbatim}

It has the five perfect matchings

\[
17263458,\quad21763458,\quad31278456,\quad
41273856,\quad51273486.                                 \tag{5.4}
\]

Their union is the entire displayed support. The edges

\[
11,12,13,14,15,21,32,53,64,75,27,58,47,88,46            \tag{5.5}
\]

form a spanning tree, so the support is connected. It consequently has total
support and is fully indecomposable. Thus deletion of the original edge
\((4,5)\) from the support augmented at \(g\) does \textbf{not} make the support
partly decomposable. By Hartfiel's definition, \(g\) is not a tie point.

The matrix \(A_\beta\) is nearly decomposable by Section 2, is the global
minimizer on the prescribed face by Section 4, and satisfies the strict
cofactor inequality by (5.3), while the conjectured conclusion fails. This
completes the proof of Theorem 1.\hfill\qedsymbol

\section*{Acknowledgments}

The proof of Theorem 1 was carried out by GPT-5.6-sol and Claude Fable 5,
under the guidance of the author. The author has reviewed the resulting proof
arguments. Responsibility for the final text rests with the author.


\begin{thebibliography}{9}

\bibitem{hartfiel1971}
D. J. Hartfiel,
\emph{On constructing nearly decomposable matrices},
Proceedings of the American Mathematical Society \textbf{27} (1971), 222--228.
\href{https://doi.org/10.1090/S0002-9939-1971-0268062-4}
{doi:10.1090/S0002-9939-1971-0268062-4}.

\bibitem{foreggersinkhorn1986}
T. H. Foregger and R. Sinkhorn,
\emph{On matrices minimizing the permanent on faces of the polyhedron of the
doubly stochastic matrices},
Linear and Multilinear Algebra \textbf{19} (1986), 395--397.
\href{https://doi.org/10.1080/03081088608817734}
{doi:10.1080/03081088608817734}.

\bibitem{foregger1987}
T. H. Foregger,
\emph{Minimum permanents of multiplexes},
Linear Algebra and its Applications \textbf{87} (1987), 197--211.
\href{https://doi.org/10.1016/0024-3795(87)90167-4}
{doi:10.1016/0024-3795(87)90167-4}.

\bibitem{minc1987}
H. Minc,
\emph{Theory of permanents 1982--1985},
Linear and Multilinear Algebra \textbf{21} (1987), 109--148.
\href{https://doi.org/10.1080/03081088708817786}
{doi:10.1080/03081088708817786}.

\bibitem{cheonwanless2005}
G.-S. Cheon and I. M. Wanless,
\emph{An update on Minc's survey of open problems involving permanents},
Linear Algebra and its Applications \textbf{403} (2005), 314--342.
\href{https://doi.org/10.1016/j.laa.2005.02.030}
{doi:10.1016/j.laa.2005.02.030}.

\end{thebibliography}
\end{document}